\documentclass[12pt,a4paper]{article}
\usepackage[utf8]{inputenc}
\usepackage[T1]{fontenc}
\usepackage{amsmath, amsthm, amsfonts, amssymb, amscd, graphicx}
\usepackage[margin=1in]{geometry}
\usepackage{float}
\usepackage{url}
\usepackage[hidelinks]{hyperref}
\usepackage{graphicx}
\usepackage{subfigure}
\usepackage{subcaption}
\usepackage{enumitem}

\newtheorem{theorem}{Theorem}[section]
\newtheorem{lemma}[theorem]{Lemma}
\newtheorem{corollary}[theorem]{Corollary}
\newtheorem{definition}[theorem]{Definition}
\newtheorem{observation}[theorem]{Observation}

\title{Self-identifying codes in direct products of complete graphs with paths and cycles}
\author{
	Jihong Liu$^{1}$, Hao Qi$^{1,*}$, Zhangwei Shan$^{1,2}$\\
	$^{1}$College of Mathematics and Physics, Wenzhou University, Wenzhou, China\\
	$^{2}$School of Mathematical Sciences, Zhejiang Normal University, Jinhua, China\\
	\and
	$^{*}$Corresponding author: \texttt{qihao@wzu.edu.cn}
}

\date{}

\begin{document}
	\maketitle
	\begin{abstract}
		\noindent
		Identifying codes were introduced by Karpovsky et al. as dominating sets $S\subseteq V(G)$ satisfying $N[u]\cap S \neq N[v]\cap S$ for any distinct vertices $u,v$. Later, Junnila et al. introduced the concept of \emph{self-identifying codes} (previously called $(1,\leq1)^+$-identifying codes in earlier work), a dominating set $S\subseteq V(G)$ such that $\bigcap_{c\in N[u]\cap S} N[c] = \{u\}$ for every vertex $u$. In this paper, we obtain bounds on the minimum size of a self-identifying code in the direct products $K_m\times P_n$ and $K_m\times C_n$ that are linear in $n$ with coefficients depending on $m$, and these bounds are asymptotically tight. In particular, for $K_m\times P_n$ with $m,n\ge3$, our bounds closely approaches the size of an identifying code in the same graph, as determined by Shinde and Waphare.
		\\[2mm]
		\\[2mm]
	\end{abstract}
	
	\baselineskip=0.20in

	\section{Introduction}\label{sec-1}
	
	Locating codes were introduced by P. J. Slater in 1988 \cite{Slater1988} in the context of fault detection in nuclear power plants. In 1998, Karpovsky, Chakrabarty, and Levitin \cite{Karpovsky1998} introduced the concept of identifying codes, which has since been studied in various graph families, including Cartesian products \cite{Goddard2013}, rook's graphs (i.e., $K_m \square K_n$) \cite{Gravier2008}, vertex-transitive graphs \cite{Gravier2015}, and binary Hamming spaces \cite{Honkala2001}. Subsequent developments include $t$-robust $1$-identifying codes \cite{Honkala2006} and complexity analyses \cite{Hudry2016}. Jean and Lobstein \cite{JeanLobstein2026} maintain a comprehensive bibliography of over 500 articles on detection systems, including identifying codes, locating-dominating codes, and their fault-tolerant variants.
	
	Junnila et al. introduced the concept of \emph{self-identifying codes} in 2020 \cite{Junnila2020} (previously called $(1,\leq1)^+$-identifying codes in earlier work \cite{Honkala2007}). Recent studies focus on self-identifying codes in specific graph families, including square grids \cite{Honkala2007}, circulants \cite{Junnila2019,Song2021}, and cubic graphs \cite{MATEMATIIKAN2022}. Very recently, Jean and Seo \cite{Jean2025} proved that determining the minimum size of a self-identifying code is NP-complete in general graphs, and studied the code in cubic graphs and infinite grids. Notably, Shinde and Waphare \cite{Shinde2023} determined the minimum size of an identifying code for $K_m \times P_n$.
	
	In this paper, we give sufficient conditions for a subset $S\subseteq V(K_m\times G)$ to be a self-identifying code (which yield an upper bound on the minimum size of such a code) and necessary conditions (which yield a lower bound), thereby deriving bounds for the minimum size of a self-identifying code in $K_m\times G$ with $G\in\{P_n,C_n\}$. Crucially, the density of the smallest self-identifying code in $K_m\times P_n$ asymptotically approaches that of the identifying code established by Shinde and Waphare \cite{Shinde2023}.
	\begin{theorem}\label{main0}
		For $m\ge 3$, the asymptotic density of the smallest self-identifying code in $K_m\times P_n$ (as $n\to\infty$, $n\ge 7$) and in $K_m\times C_n$ (as $n\to\infty$, $n\ge 3$) is $1/3$.
	\end{theorem}
	
	The paper is organized as follows. Section~\ref{sec-2} introduces the necessary notation and preliminaries. In Section~\ref{sec-3}, we present the bounds for $K_m\times P_n$, and in Section~\ref{sec-4} we treat the case $K_m\times C_n$. The proofs of the upper bounds rely on explicit constructions (the sufficient conditions), while the lower bounds follow from combinatorial constraints (the necessary conditions). The density result is discussed in the final part of each section.

	\section{Terminology and Notation}\label{sec-2}
	
	A simple undirected graph $G$ is defined as an ordered pair $(V(G), E(G))$, where $V(G)$ is the vertex set and $E(G)$ the edge set. 
	Given graphs $G$ and $H$, their \emph{direct product} $G \times H$ (Figure~\ref{fig:directproduct}) has vertex set $V(G) \times V(H)$ and edge set:
	\[
	E(G \times H) = \left\{ (g_1, h_1)(g_2, h_2) \mid g_1g_2 \in E(G) \text{ and } h_1h_2 \in E(H) \right\}.
	\]
	For clarity, we adopt the following notation (Figure~\ref{fig:directproduct}):
	\begin{itemize}
		\item $V(K_m) = \{v_0, v_1, \ldots, v_{m-1}\}$, $V(P_n) = \{0,1,\ldots,n-1\}$, and $V(C_n) = \{0,1,\ldots,n-1\}$ with edges between consecutive indices modulo $n$.
		\item The $i$-th \emph{row} $R_i$ in $K_m \times P_n$ (or $K_m \times C_n$) is $R_i = \{(v_i, j) \mid j \in V(P_n)\}$ (or $j\in V(C_n)$).
		\item The $j$-th \emph{column} $C_j$ in $K_m \times P_n$ (or $K_m \times C_n$) is $C_j = \{(v_i, j) \mid v_i \in V(K_m)\}$.
	\end{itemize}
	
	The \emph{open neighborhood} $N(v)$ of a vertex $v$ is its set of adjacent vertices, and the \emph{closed neighborhood} is $N[v] = \{v\} \cup N(v)$. 
	A nonempty subset $D \subseteq V(G)$ is called a \emph{dominating set} if $N[v] \cap D \neq \emptyset$ for every vertex $v \in V(G)$. 
	A nonempty subset $S \subseteq V(G)$ is called a \emph{separating set} if $N[u] \cap S \neq N[v] \cap S$ for all distinct vertices $u, v \in V(G)$. 
	A nonempty subset $C \subseteq V(G)$ is called an \emph{identifying code} if it is both a dominating set and a separating set.
	
	\begin{definition}\label{def:SID1}
		A nonempty subset $S \subseteq V(G)$ is a \emph{self-identifying code} if for each $v \in V(G)$, $N[v] \cap S \neq \emptyset$, and
		\begin{equation}
			\bigcap_{c \in N[v] \cap S} N[c] = \{v\}.
		\end{equation}
	\end{definition}
	
	The following equivalent definition is given in \cite{Junnila2020}.
	
	\begin{definition}\label{def:SID2}
		A nonempty subset $S \subseteq V(G)$ is a self-identifying code if for all distinct $u, v \in V(G)$,
		\begin{equation}
			(N[u] \cap S) \setminus (N[v] \cap S) \neq \emptyset.
		\end{equation}
	\end{definition}
	
	\begin{figure}[H]
		\centering
		\subfigure[$K_3\times P_4$]{
			\label{KPdirectproduct}
			\includegraphics[width=4.5cm]{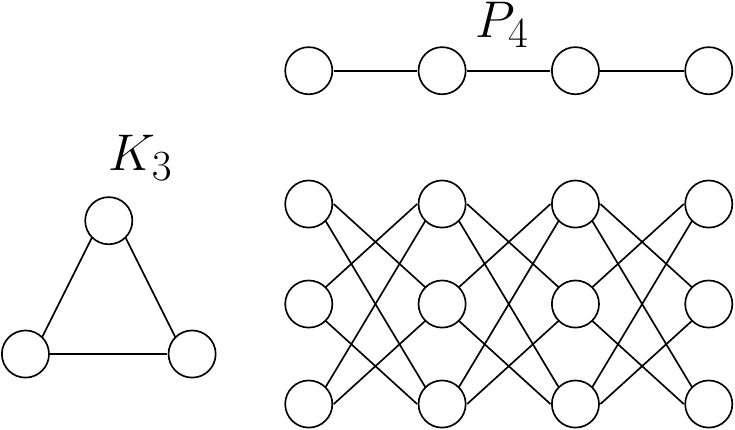}}
		\hspace{0.5in}
		\subfigure[$K_5\times P_5$]{
			\label{self-identifyingexample}
			\includegraphics[width=3.5cm]{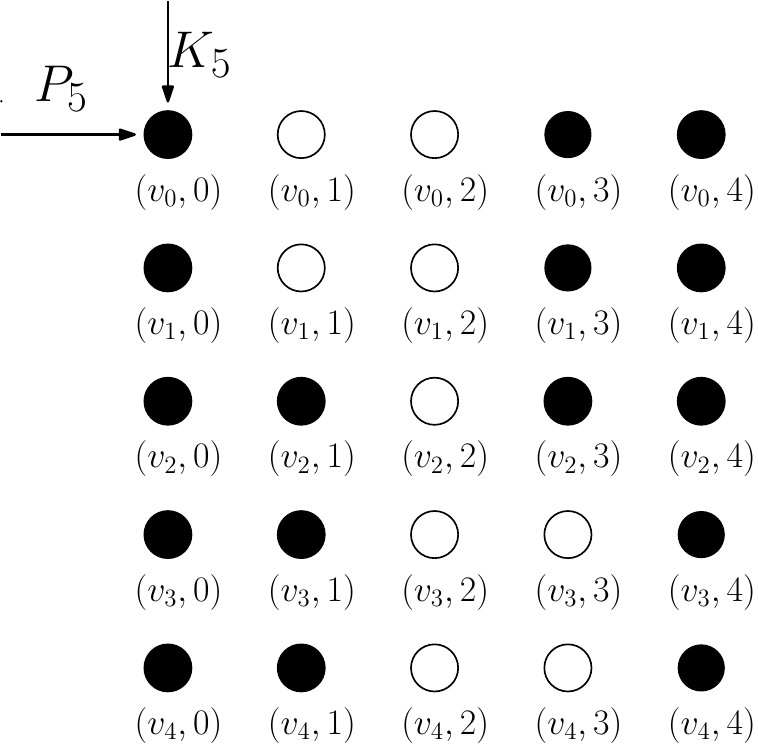}}
		\caption{The direct product $K_3\times P_4$, and a self-identifying code of $K_5\times P_5$ (black vertices).}\label{fig:directproduct}
	\end{figure}
	In figures, vertices in $S$ are marked as black circles. Edges are omitted for clarity in subsequent figures. Undefined terms follow Bondy \cite{Bondy08} and West \cite{West2001}. Let $[a, b] = \{a, a+1, \ldots, b\}$.
	
	The minimum sizes of identifying and self-identifying codes in $G$ are denoted $\gamma^{\mathrm{ID}}(G)$ and $\gamma^{\mathrm{SID}}(G)$, respectively. Then we have $\gamma^{\mathrm{SID}}(G) \ge \gamma^{\mathrm{ID}}(G)$. 
	If $G$ has a self-identifying code, then $G$ is called a \emph{self-identifiable graph}. 
	The following characterization of self-identifiable graphs is given in \cite{Junnila2020}.
	
	\begin{lemma}[\cite{Junnila2020}]\label{lem:self-identifiable-condition}
		A graph $G$ is self-identifiable if and only if for all distinct $u, v \in V(G)$,
		\begin{equation}\label{eq:graph-condition}
			N[u] \setminus N[v] \neq \emptyset.
		\end{equation}
	\end{lemma}
	
	Note that complete graphs $K_m$ (for $m\ge 2$) are not self-identifiable, since all vertices have identical closed neighborhoods.

	\begin{lemma}\label{lem:degree-condition}
		If a connected graph $G$ admits a self-identifying code $S$, then $|N(v) \cap S| \ge 2$ for all $v \in V(G)$.
	\end{lemma}
	\begin{proof}
		Suppose $|N(u) \cap S| \le 1$ for some $u \in V(G)$. If $|N(u) \cap S| = 1$, let $N(u) \cap S = \{w\}$. Then $\{u, w\} \subseteq \bigcap_{c \in N[u] \cap S} N[c]$, contradicting Definition~\ref{def:SID1}. If $|N(u) \cap S| = 0$, then $u \in S$ implies $\bigcap_{c \in N[u] \cap S} N[c] = N[u] \neq \{u\}$, as $G$ has no isolated vertices.
	\end{proof}
	
	Lemma~\ref{lem:degree-condition} implies that graphs with leaves (e.g., $P_n$, $K_1 \times P_n$, $K_2 \times P_n$) are not self-identifiable. Thus, we focus on $K_m \times P_n$ and $K_m \times C_n$ for $m, n \ge 3$.
	
	\begin{lemma}\label{lem:product-self-identifiable}
		For $m,n\ge 3$, both $K_m \times P_n$ and $K_m \times C_n$ are self-identifiable graphs.
	\end{lemma}
	\begin{proof}
		By Lemma~\ref{lem:self-identifiable-condition}, it suffices to show $N[u] \setminus N[w] \neq \emptyset$ for any distinct $u,w\in V(K_m\times P_n)$. Let $u=(v_i,j)$ and $w=(v_{i'},j')$. If $j=j'$, then $u$ and $w$ are in the same column but not adjacent since they differ only in the first coordinate, they are not adjacent; choosing a neighbor of $u$ in an adjacent column yields a vertex in $N[u]\setminus N[w]$. If $j\neq j'$, one can similarly find a suitable neighbor. The proof for $K_m\times C_n$ is analogous, taking indices modulo $n$. 
	\end{proof}
	
	Our main results bound $\gamma^{\mathrm{SID}}$ for these products:
	
	\begin{theorem}\label{thm:Pn-bound}
		For $m \ge 3$ and $n \ge 7$,
		\[
		\left\lceil \frac{n+1}{3} \right\rceil (m + 2) - 2 \le \gamma^{\mathrm{SID}}(K_m \times P_n) \le \left\lceil \frac{n+1}{3} \right\rceil (m + 3) + m.
		\]
	\end{theorem}
	
	\begin{theorem}\label{thm:Cn-bound}
		For $m, n \ge 3$,
		\[
		\left\lfloor \frac{n}{3} \right\rfloor (m + 2) \le \gamma^{\mathrm{SID}}(K_m \times C_n) \le \left\lceil \frac{n}{3} \right\rceil (m + 3) + 3.
		\]
	\end{theorem}
	In the following two sections we establish the lower and upper bounds for $K_m\times P_n$ and $K_m\times C_n$, respectively. We begin with the path case.
	\section{Bounds on Self-Identifying Codes of $K_m \times P_n$}\label{sec-3}
	
	In this section, we first present several useful lemmas, corollaries, and preliminary results in Subsection~\ref{nessaryconditionKmPn}, which establish the lower bound. Subsequently, in Subsection~\ref{conboundKmPn}, we construct a self-identifying code for $K_m \times P_n$ with $m,n \geq 3$ to derive the upper bound. 
	
	Recall the definitions: for $0 \leq i \leq m-1$ and $0 \leq j \leq n-1$, we define
	\begin{itemize}
		\item $R_i = \{(v_i,j) : j \in V(P_n)\}$ as the $i$-th row,
		\item $C_j = \{(v_i,j) : v_i \in V(K_m)\}$ as the $j$-th column.
	\end{itemize}
	\subsection{Necessary Conditions for a Self-Identifying Code in $K_m \times P_n$}\label{nessaryconditionKmPn}
	
	This subsection establishes properties of self-identifying codes in $K_m \times P_n$ using row and column structures. Let $S$ denote a self-identifying code in $K_m \times P_n$ with $m, n \geq 3$.

	\begin{lemma}\label{c0cn-1}
		The boundary columns satisfy $C_0 \subseteq S$ and $C_{n-1} \subseteq S$.  
	\end{lemma}
	\begin{proof}
		Assume that there exists $i\in[0,m-1]$ such that $(v_i,0)\in C_0\setminus S$. If there exists $i'\neq i$ with $(v_{i'},1)\notin S$, then $\{(v_i,1),(v_{i'},1)\}\subseteq\bigcap_{c\in N[(v_{i'},1)]\cap S}N[c]$, contradicting Definition~\ref{def:SID1}. Hence $C_1\setminus\{(v_i,1)\}\subseteq S$. Then 
		$N[(v_i,0)]\cap S\subseteq N[(v_i,2)]\cap S$, contradicting Definition~\ref{def:SID2}. Therefore $C_0\subseteq S$. The same argument gives $C_{n-1}\subseteq S$.
	\end{proof}
	By Lemma~\ref{c0cn-1} and Lemma~\ref{lem:degree-condition}, we have the following corollary.
	\begin{corollary}\label{c1cn-2}
		The near-boundary columns satisfy $|C_1 \cap S| \geq 3$ and $|C_{n-2} \cap S| \geq 3$.  
	\end{corollary}
	
	Lemmas~\ref{c0cn-1} ensure boundary constraints, while Corollary~\ref{c1cn-2} gives columns near boundaries. For internal columns, we derive
	
	\begin{lemma}\label{white}
		Let $n \geq 5$, $i \in [0, m-1]$, $j \in [2, n-3]$, and $(v_i, j) \notin S$. Then,
		\begin{enumerate}[font = \normalfont, label=(\arabic*)]
			\item $S \cap (C_{j-1} \cup C_{j+1}) \cap R_{i'} \neq \emptyset$ for all $i' \neq i$, i.e., $|N[(v_i, j)] \cap S| \geq m - 1$, 
			\item $S \cap C_{j-1} \neq \emptyset$ and $S \cap C_{j+1} \neq \emptyset$.
		\end{enumerate}
	\end{lemma}
	\begin{proof}
		Since $(v_i,j)\notin S$, we have $N[(v_i,j)]\subseteq C_{j-1}\cup C_{j+1}$. If there exists $i'\neq i$ such that $\{(v_{i'},j-1),(v_{i'},j+1)\}\cap S=\emptyset$, then $\{(v_i,j),(v_{i'},j)\}\subseteq\bigcap_{c\in N[(v_i,j)]\cap S}N[c]$, contradicting Definition~\ref{def:SID1}. Hence for every $i'\neq i$, we have $S\cap(C_{j-1}\cup C_{j+1})\cap R_{i'}\neq\emptyset$, which gives $|N[(v_i,j)]\cap S|\ge m-1$. If $C_{j-1}\cap S=\emptyset$, then $N[(v_i,j)]\cap S\subseteq N[(v_i,j+2)]\cap S$, contradicting Definition~\ref{def:SID2}. Similarly, $S\cap C_{j+1}\neq\emptyset$.
	\end{proof}
	By Lemma~\ref{white} and Lemma~\ref{lem:degree-condition}, we have the following corollary.
	\begin{corollary}\label{mainwhitec} 
		For $n \geq 5$ and $j \in [2, n-3]$:
		\begin{enumerate}[font = \normalfont, label=(\arabic*)]
			\item If $j \in [3, n-4]$ and $C_j \cap S = \emptyset$, then $(C_{j-1} \cup C_{j+1}) \subseteq S$ for $n \geq 7$.
			\item If $|C_j \cap S| = 1$:
			\begin{itemize}
				\item  For $j=2$, then $|(C_{1}\cup C_{3})\cap S|\ge m$,
				\item For $j=n-3$, then $|(C_{n-4}\cup C_{n-2})\cap S|\ge m$;\label{whitec1}
				\item For $j \in [3, n-4]$, $|(C_{j-1} \cup C_{j+1}) \cap S| \geq m + 1$.
			\end{itemize}
			\item If $2 \leq |C_j \cap S| \leq m - 2$, then $|(C_{j-1} \cup C_{j+1}) \cap S| \geq m$ and $(C_{j-1} \cup C_{j+1}) \cap S \cap R_i \neq \emptyset$ for all $i$.
			\item If $|C_j \cap S| = m - 1$, then $|(C_{j-1} \cup C_{j+1}) \cap S| \geq m - 1$.
			\item If $|C_j \cap S| = m$, then $|(C_{j-1} \cup C_{j+1}) \cap S| \geq 3$.
		\end{enumerate}
	\end{corollary}
	By Corollary~\ref{mainwhitec}, there exists a special lower bound for three consecutive columns.
	\begin{corollary}\label{cj-1cjcj+1}
		Let $n \geq 5$, $i \in [0, m-1]$, $j \in [2, n-3]$, then
		\begin{enumerate}[font = \normalfont, label=(\arabic*)]
			\item For $n = 5$, $|(C_1 \cup C_2 \cup C_3) \cap S| \geq m$.
			\item For $n \geq 6$, $j \in [2, n-3]$, $|(C_{j-1} \cup C_j \cup C_{j+1}) \cap S| \geq m + 1$. 
			\item For $n \geq 7$ and $j \in [3, n-4]$, $|(C_{j-1} \cup C_j \cup C_{j+1}) \cap S| \geq m + 2$. 
		\end{enumerate}
	\end{corollary}
	
	\begin{proof}[Proof of Theorem~\ref{thm:Pn-bound} {\rm (}Lower Bound{\rm )}]
		By Lemma~\ref{c0cn-1},  Corollary~\ref{c1cn-2}, and Corollary~\ref{cj-1cjcj+1}, then
		\begin{align*}
			|(C_0\cup C_1\big)\cap S|&\ge m+3, \\
			|(C_{n-2}\cup C_{n-1}) \cap S| &\geq m + 3,\\
			|(C_{j-1} \cup C_j \cup C_{j+1}) \cap S| &\geq m + 2.
		\end{align*}
		Therefore,
		\[
		\gamma^{\mathrm{SID}}(K_m \times P_n) \geq \left\lceil \frac{n+1}{3} \right\rceil (m + 2) - 2.
		\]
	\end{proof}
	Having established a lower bound via necessary conditions, we now turn to an upper bound by constructing an explicit self-identifying code.
	
	\subsection{Construction of a Self-Identifying Code in $K_m \times P_n$}\label{conboundKmPn}
	In this subsection we construct a self-identifying code for $K_m \times P_n$, which yields an upper bound on $\gamma^{\mathrm{SID}}(K_m \times P_n)$ for $m \ge 3$ and $n \ge 7$. We begin with a sufficient condition for self-identifying codes.
	
	\begin{observation}\label{factKmPn}
		Let $S \subseteq V(K_m \times P_n)$. If the following conditions hold for all $i\in[0,m-1]$ and $j\in[0,n-1]$,
		\begin{itemize}
			\item $|C_1 \cap S| \geq 3$ and $|C_{n-2} \cap S| \geq 3$;
			\item For $j \in [1, n-2]$ and for each vertex $(v_i,j) \in C_j$:
			\begin{itemize}
				\item if $(v_i,j) \in S$, then there exist two distinct indices $i_1,i_2 \in [0,m-1]\setminus\{i\}$ such that $(C_{j-1} \cup C_{j+1}) \cap S \cap R_{i'} \neq \emptyset$ for $i' \in \{i_1,i_2\}$;
				\item if $(v_i,j) \notin S$, then there exist two distinct indices $i_1,i_2 \in [0,m-1]\setminus\{i\}$ such that $(C_{j-1} \cup C_{j+1}) \cap S \cap R_{i'} \neq \emptyset$, and additionally $C_{j-1} \cap S \neq \emptyset$, $C_{j+1} \cap S \neq \emptyset$;
			\end{itemize}
		\end{itemize}
		then $S$ is a self-identifying code of $K_m \times P_n$.
	\end{observation}
	
	By Observation~\ref{factKmPn}, we construct an upper bound of $\gamma^{\mathrm{SID}}(K_m \times P_n)$ in Theorem~\ref{thm:Pn-bound}.
	
	\begin{proof}[Proof of Theorem~\ref{thm:Pn-bound} {\rm (}Upper Bound{\rm )}] 
		For non-negative integers $m,n,t,k$, and $k=\lfloor \frac{n}{3}\rfloor$. 
		We define 
		\[
		\mathcal C_t = 
		\begin{cases}
			C_{3t-1} \cup C_{3t} \cup C_{3t+1} & \text{for } 1 \leq t \leq k - 2, \\
			C_{n-4} \cup C_{n-3} & \text{for $t=k-1$  and } n = 3k, \\
			C_{n-5} \cup C_{n-4} \cup C_{n-3} & \text{for $t=k-1$  and } n = 3k + 1, \\
			C_{n-6} \cup C_{n-5} \cup C_{n-4} \cup C_{n-3} & \text{for $t=k-1$  and } n = 3k + 2.
		\end{cases}
		\]
		For the partition $V(K_m\times P_n)=C_0\cup C_1\cup \left(\bigcup_{t=1}^{k-2} \mathcal C_t\right)\cup \mathcal C_{k-1}\cup C_{n-2} \cup C_{n-1}$, we construct a subset $S$ of $V(K_m\times P_n)$ with  $S = S_1 \cup S_2 \cup S_3 \cup S_4$, where 
		\begin{equation*}
			\begin{aligned}
				S_1 &= ( C_0 \cup C_{n-1})\cap S = C_0 \cup C_{n-1}, \\
				S_2 &= ( C_1 \cup C_{n-2})\cap S=\{(v_0, 1), (v_1, 1), (v_2, 1), (v_0, n-2), (v_1, n-2), (v_2, n-2)\}, \\
				S_3 &=\left(\bigcup\nolimits_{t=1}^{k-2} \mathcal C_t\right) \cap S =
				\begin{cases}
					\bigcup_{t=1,3,5,\dots,k-3} (A_t \cup C_{3t} \cup C_{3t+3}) & \text{if } k \text{ even}, \\
					\left(\bigcup_{t=1,3,5,\dots,k-4} A'_t \cup C_{3t} \cup C_{3t+3}\right) \cup A'' \cup C_{3k-6} & \text{if } k \text{ odd},
				\end{cases}\\
				S_4 &=\mathcal C_{k-1}\cap S= 
				\begin{cases}
					\{(v_0, n-3), (v_1, n-3), (v_2, n-3)\} \cup C_{n-4} & \text{if } n = 3k, \\
					\{(v_2, n-5), (v_0, n-3), (v_1, n-3)\} \cup C_{n-4} & \text{if } n = 3k + 1, \\
					\{(v_2, n-6), (v_0, n-4), (v_1, n-4)\} \cup C_{n-5} \cup C_{n-3} & \text{if } n = 3k + 2,
				\end{cases} 
			\end{aligned}
		\end{equation*}
		\begin{align*}
			A_t &= \{(v_2, 3t-1), (v_0, 3t+1), (v_1, 3t+1), (v_0, 3t+2), (v_1, 3t+4), (v_2, 3t+4)\},\\
			A'_t &= \{(v_0, 3t-1), (v_1, 3t+1), (v_2, 3t+1), (v_2, 3t+2), (v_0, 3t+4), (v_1, 3t+4)\},\\
			A''&=\{(v_0, 3k-7), (v_1, 3k-5), (v_2, 3k-5)\}.
		\end{align*}
		By Observation~\ref{factKmPn}, $S$ is a self-identifying code of $K_m \times P_n$ (Figures~\ref{n=3k,3k+1,3k+2even} and~\ref{n=3k,3k+1,3k+2odd}). Thus
		\[
		\gamma^{\mathrm{SID}}(K_m \times P_n) \leq 
		\begin{cases}
			(k + 1)(m + 3) & \text{if } n = 3k, \\
			(k + 1)(m + 3) & \text{if } n = 3k + 1, \\
			(k + 1)(m + 3) + m & \text{if } n = 3k + 2.
		\end{cases}
		\]
		\begin{figure}[H]	
			\centering
			\subfigure[$n=3k$, $k$ is even]{
				\label{n=3keven} 
				\includegraphics[width=3.15cm]{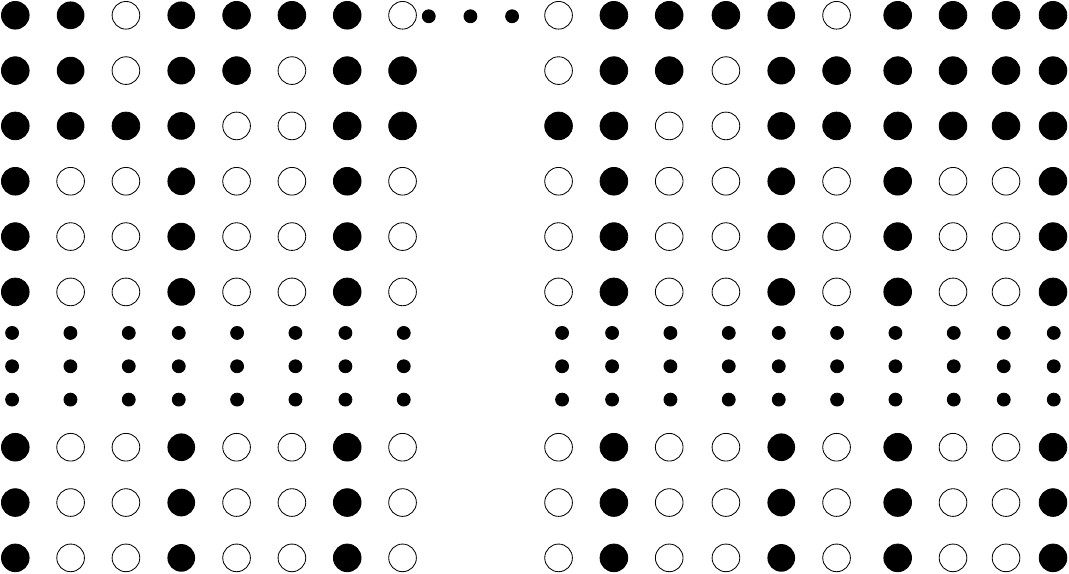}}
			\hspace{0.15in}
			\subfigure[$n=3k+1$, $k$ is even]{
				\label{n=3k+1even} 
				\includegraphics[width=3.35cm]{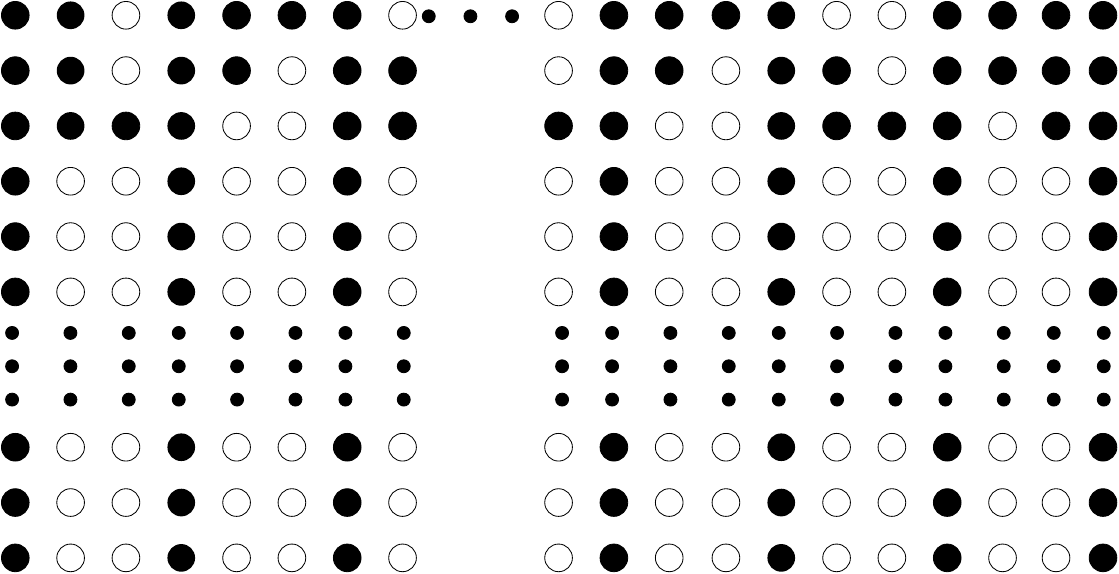}}
			\hspace{0.15in}
			\subfigure[$n=3k+2$, $k$ is even]{
				\label{n=3k+2even} 
				\includegraphics[width=3.75cm]{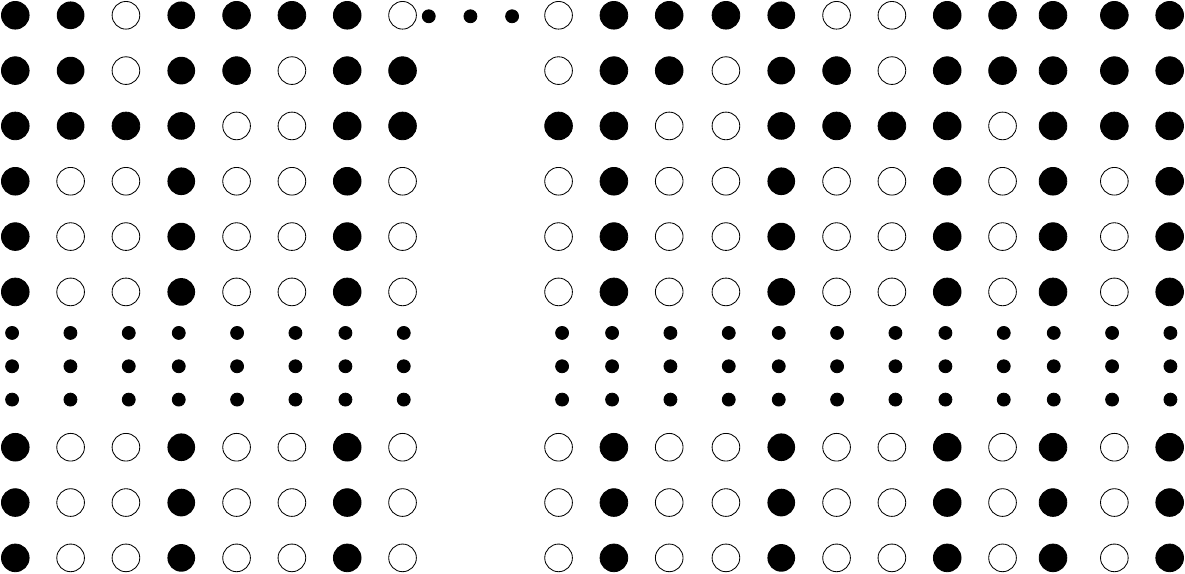}}
			\caption{A self-identifying code of $K_m\times P_n$, where $n$ is even.}\label{n=3k,3k+1,3k+2even}
		\end{figure}
		\begin{figure}[H]	
			\centering
			\subfigure[$n=3k$, $k$ is odd]{
				\label{n=3kodd} 
				\includegraphics[width=3.15cm]{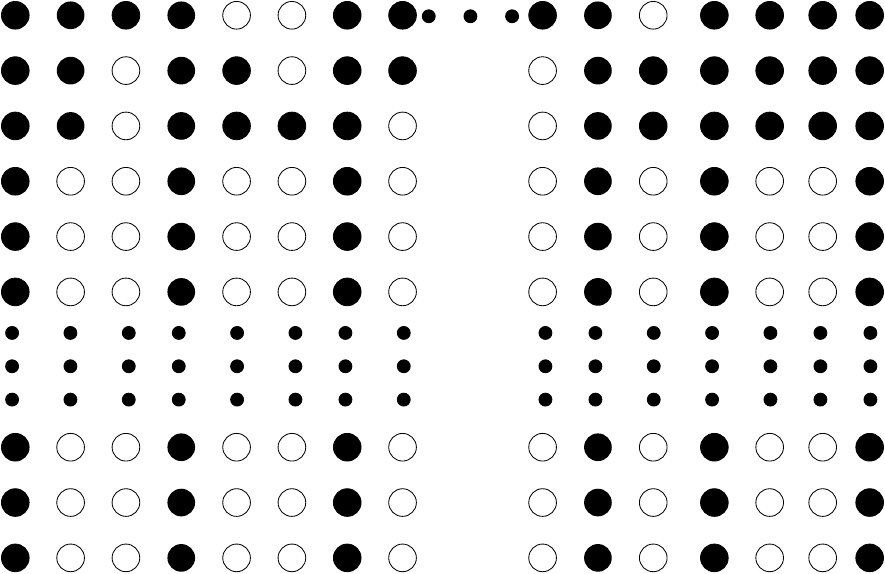}}
			\hspace{0.15in}
			\subfigure[$n=3k+1$, $k$ is odd]{
				\label{n=3k+1odd} 
				\includegraphics[width=3.35cm]{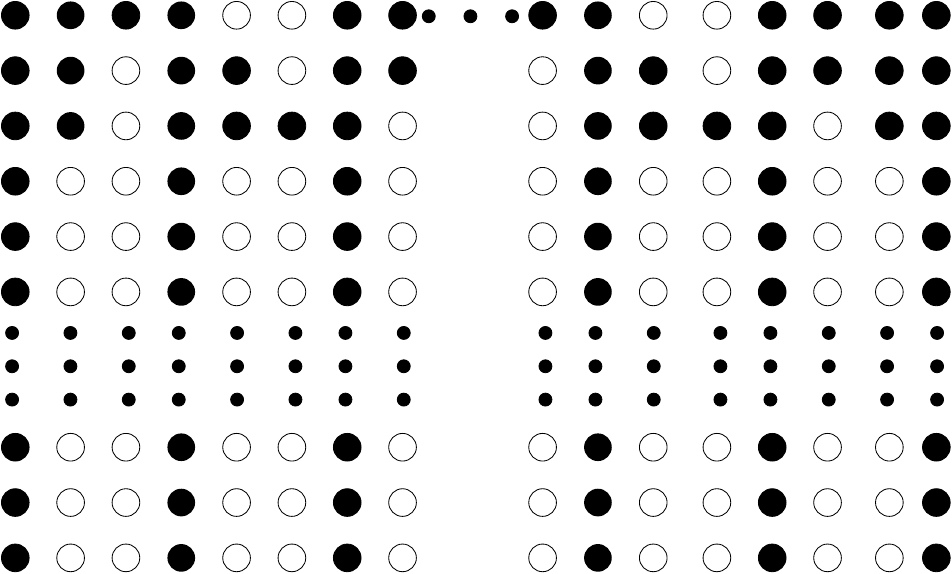}}
			\hspace{0.15in}
			\subfigure[$n=3k+2$, $k$ is odd]{
				\label{n=3k+2odd} 
				\includegraphics[width=3.75cm]{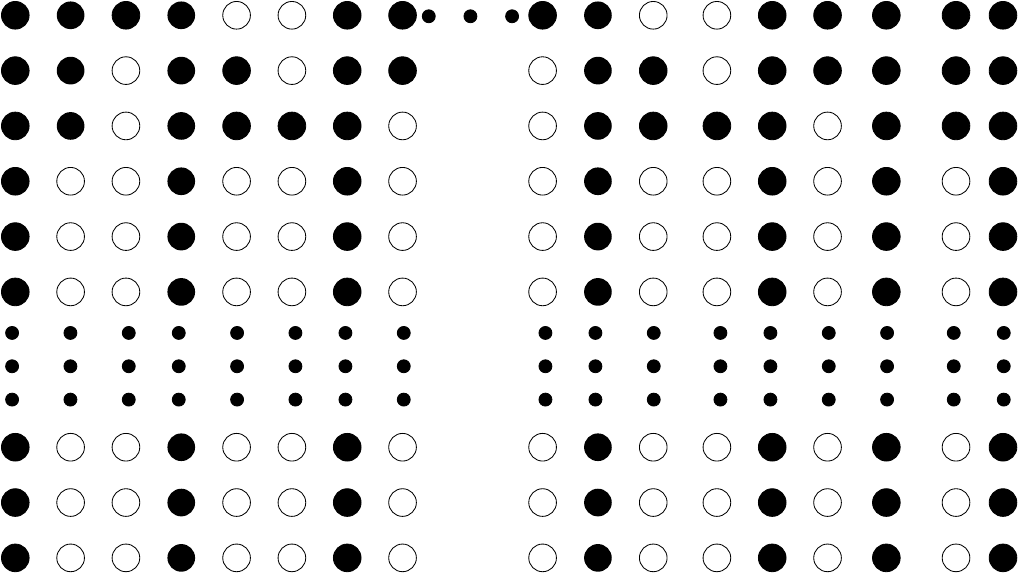}}
			\caption{A self-identifying code of $K_m\times P_n$, where $n$ is odd.}\label{n=3k,3k+1,3k+2odd}
		\end{figure}
		The upper bound of $\gamma^{\mathrm{SID}}(K_m\times P_n)$ for $m\ge 3$ and $3\le n\le 6$, due to its technical intricacy, is deferred to the Appendix~A.
		Hence,  we prove Theorem~\ref{thm:Pn-bound}.
	\end{proof}
	
	We now consider the direct product of a complete graph $K_m$ with a cycle $C_n$, and obtain analogous bounds.

	\section{Bounds on Self-Identifying Codes of $K_m \times C_n$}\label{sec-4}
	In this section, we consider self-identifying codes in $K_m\times C_n$ for $m,n\ge 3$. Using a similar method as in Section~\ref{sec-3}, we derive lower and upper bounds for $\gamma^{\mathrm{SID}}(K_m\times C_n)$, which together establish Theorem~\ref{thm:Cn-bound}. We use the same notation for rows and columns as before: $R_i = \{(v_i,j): j\in V(C_n)\}$ (the $i$-th row) and $C_j = \{(v_i,j): v_i\in V(K_m)\}$ (the $j$-th column). In this section, all column indices are taken modulo $n$. 
	
	Let us assume $S$ is a self-identifying code of $K_m \times C_n$, where $m,n\ge 3$. The analysis is similar to the path case; in analogy with Corollary~\ref{cj-1cjcj+1}, we obtain the following lemma, which leads to the lower bound.
	\begin{lemma}\label{cj-1cjcj+1KmCn}
		For any $0 \leq j \leq n-1$, $|(C_{j-1} \cup C_j \cup C_{j+1}) \cap S| \geq m + 2$.
	\end{lemma}
	\begin{corollary}[Lower Bound in Theorem~\ref{thm:Cn-bound}]\label{KmCnlowerbound}
		Let $m,n\ge 3$, then
		\[\gamma^{\mathrm{SID}}(K_m \times C_n) \geq \left\lfloor \frac{n}{3} \right\rfloor (m + 2)\] 
	\end{corollary}
	Now we will construct a self-identifying code for $K_m \times C_n$, which yields an upper bound on $\gamma^{\mathrm{SID}}(K_m \times C_n)$ for $m,n \geq 3$. We begin with a sufficient condition for self-identifying codes as follows.
	\begin{observation}\label{factKmCn}
		Let $S \subseteq V(K_m \times C_n)$. If the following conditions hold for all $i \in [0, m-1]$ and $j \in [0, n-1]$,
		\begin{itemize}
			\item For $(v_i, j) \in S$, there exist two distinct indices $i_1, i_2\in [0,m-1]\setminus \{i\}$ such that $(C_{j-1} \cup C_{j+1}) \cap S \cap R_{i'} \neq \emptyset$ for $i' \in \{i_1, i_2\}$.
			\item For $(v_i, j) \notin S$, there exist two distinct indices $i_1,i_2\in[0,m-1]\setminus\{i\}$ such that $(C_{j-1} \cup C_{j+1}) \cap S \cap R_{i'} \neq \emptyset$, and $C_{j-1} \cap S \neq \emptyset$, $C_{j+1} \cap S \neq \emptyset$.
		\end{itemize}
		Then $S$ is a self-identifying code of $K_m \times C_n$.
	\end{observation}
	
	By Observation~\ref{factKmCn}, we construct an upper bound of $\gamma^{\mathrm{SID}}(K_m \times C_n)$ in Theorem~\ref{thm:Cn-bound}.
	
	For non-negative integers $m,n,t,k$, and $k=\lfloor \frac{n}{3}\rfloor$. Define a nonempty set $S \subseteq V(K_m\times C_n)$ as follows. Let 
	\begin{align*}
		T=\{0,2,4,6,\dots,k-2\} & ~~~~~~\text{if $k$ is even}\\
		T'=\{0,2,4,6,\dots,k-3\} & ~~~~~~\text{if $k$ is odd}
	\end{align*}
	If $k$ is even, then let
	\[
	S = 
	\begin{cases}
		\bigcup_{t\in T} (B_{t} \cup C_{3t+1} \cup C_{3t+4}) & \text{if } n = 3k, \\
		\bigcup_{t\in T} (B_{t} \cup C_{3t+1} \cup C_{3t+4}) \cup C_{3k} & \text{if } n = 3k + 1, \\
		\bigcup_{t\in T} (B_{t} \cup C_{3t+1} \cup C_{3t+4}) \cup B' \cup C_{3k+1} & \text{if } n = 3k + 2,
	\end{cases}
	\]
	If $k$ is odd, then let
	\[
	S = 
	\begin{cases}
		\bigcup_{t\in T'} (B_{t} \cup C_{3t+1} \cup C_{3t+4}) \cup B'' \cup C_{3k-2} & \text{if } n = 3k, \\
		\bigcup_{t\in T'} (B_{t} \cup C_{3t+1} \cup C_{3t+4}) \cup B'' \cup C_{3k-2} \cup C_{3k} & \text{if } n = 3k + 1, \\
		\bigcup_{t\in T'} (B_{t} \cup C_{3t+1} \cup C_{3t+4}) \cup B'' \cup C_{3k-2} \cup B' \cup C_{3k+1} & \text{if } n = 3k + 2,
	\end{cases}
	\]
	where \begin{align*}
		B_t &= \{(v_0, 3t), (v_1, 3t+2), (v_2, 3t+2), (v_2, 3t+3), (v_0, 3t+5), (v_1, 3t+5)\}\\
		B'&=\{(v_0, 3k), (v_1, 3k), (v_2, 3k)\}\\
		B''&=\{(v_0, 3k-3), (v_1, 3k-3), (v_2, 3k-3),(v_0, 3k-1), (v_1, 3k-1), (v_2, 3k-1)\}.
	\end{align*}
	\begin{figure}[H]	
		\centering
		\subfigure[$n=3k$, $k$ is even]{
			\label{n=3kevenCn} 
			\includegraphics[width=3.15cm]{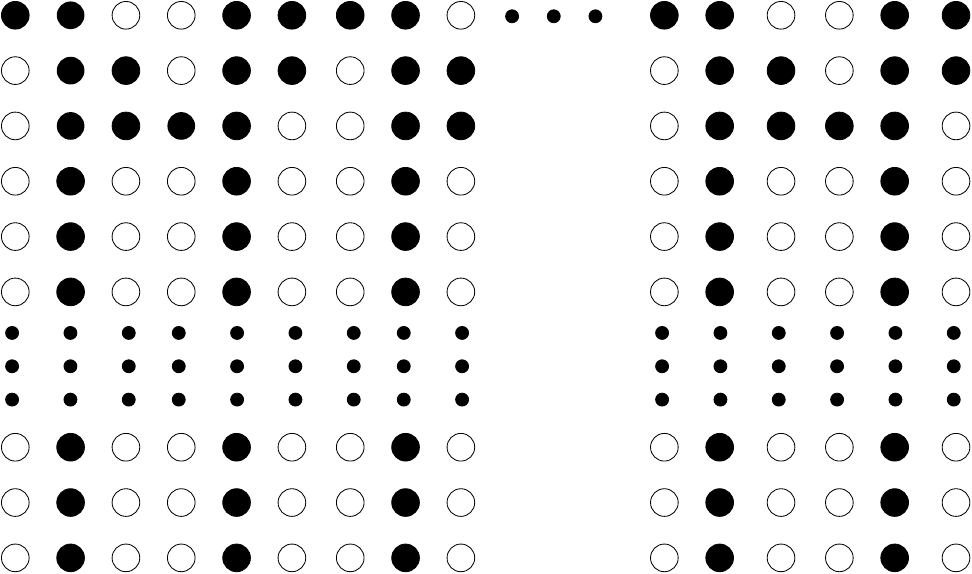}}
		\hspace{0.15in}
		\subfigure[$n=3k+1$, $k$ is even]{
			\label{n=3k+1evenCn} 
			\includegraphics[width=3.35cm]{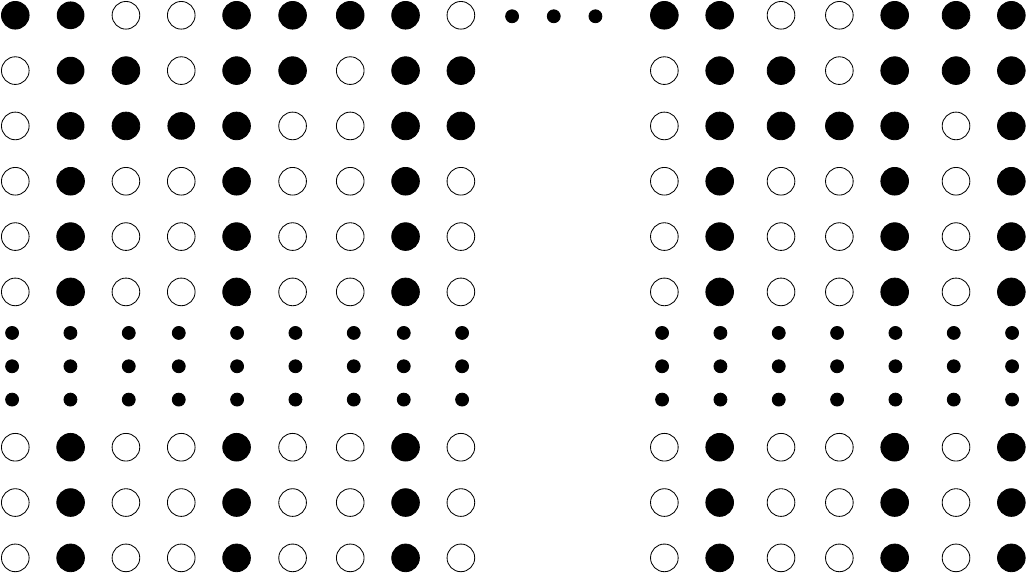}}
		\hspace{0.15in}
		\subfigure[$n=3k+2$, $k$ is even]{
			\label{n=3k+2evenCn} 
			\includegraphics[width=3.75cm]{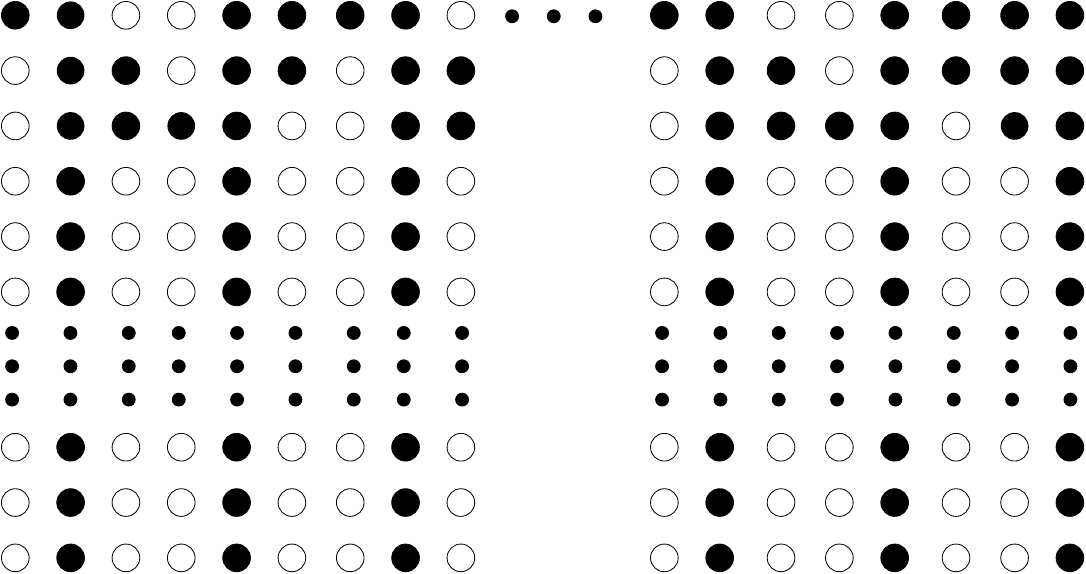}}
		\caption{A self-identifying code of $K_m\times C_n$, where $n$ is even.}\label{n=3k,3k+1,3k+2evenCn}
	\end{figure}
	\begin{figure}[H]	
		\centering
		\subfigure[$n=3k$, $k$ is odd]{
			\label{n=3koddCn} 
			\includegraphics[width=3.15cm]{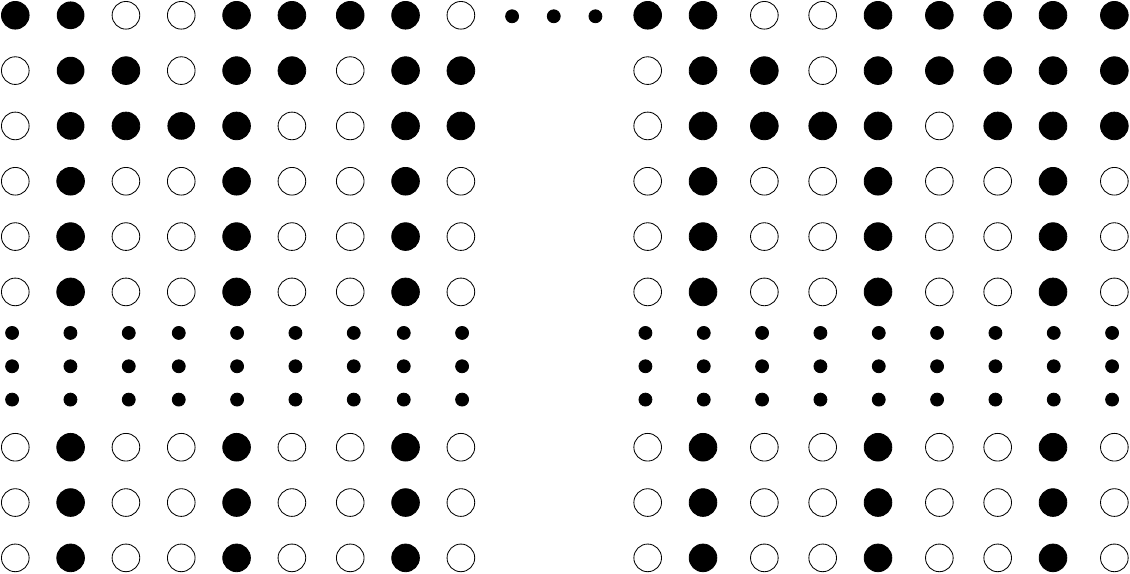}}
		\hspace{0.15in}
		\subfigure[$n=3k+1$, $k$ is odd]{
			\label{n=3k+1oddCn} 
			\includegraphics[width=3.35cm]{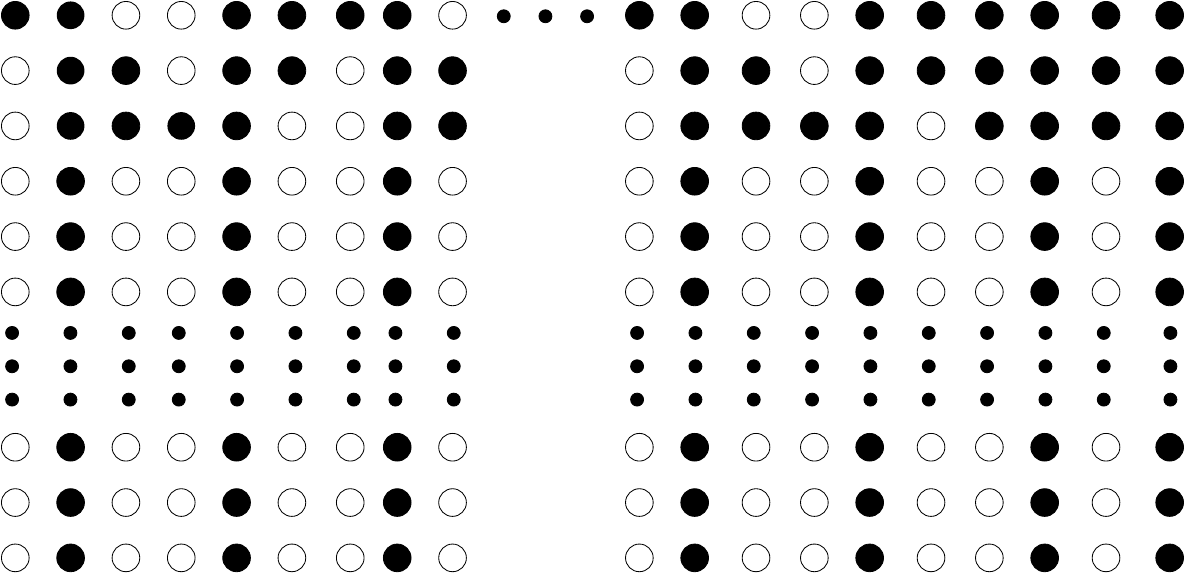}}
		\hspace{0.15in}
		\subfigure[$n=3k+2$, $k$ is odd]{
			\label{n=3k+2oddCn} 
			\includegraphics[width=3.75cm]{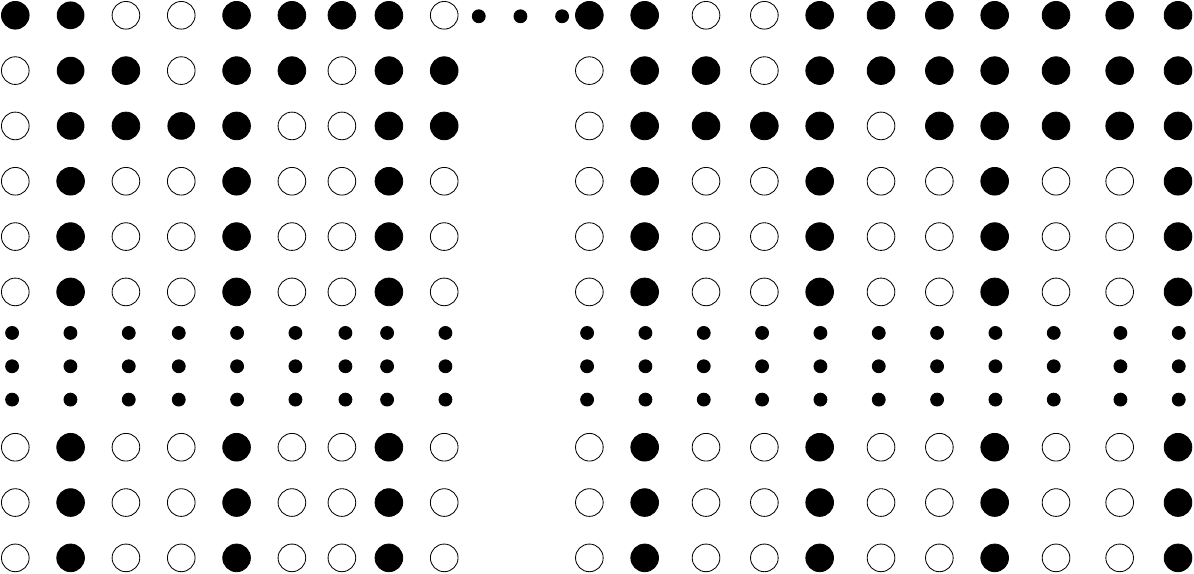}}
		\caption{A self-identifying code of $K_m\times C_n$, where $n$ is odd.}\label{n=3k,3k+1,3k+2oddCn}
	\end{figure}
	By Observation~\ref{factKmCn}, $S$ is a self-identifying code of $K_m \times C_n$. Thus for $m, n \geq 3$, 
	\[
	\gamma^{\mathrm{SID}}(K_m \times C_n) \leq \left\lceil \frac{n}{3} \right\rceil (m + 3) + 3
	\]
	Hence, we prove Theorem~\ref{thm:Cn-bound}.
	
	\section*{Acknowledgment}
	The authors thank the anonymous referees for their insightful suggestions, which improved the quality of this work. This research was supported by the Zhejiang Provincial Natural Science Foundation of China (Grant No. LQ24A010014) and the Zhejiang Provincial College Students’ Science and Technology Innovation Activity Plan (Xinmiao Talent Program) (Grant Nos. 2023R451014, 2024R429A011).


	\section*{Appendix A: Self-Identifying Codes in $K_m \times P_n$ for Small $n$}
	\label{app:A}
	This appendix establishes exact values of $\gamma^{\mathrm{SID}}(K_m \times P_n)$ for $m \geq 3$ and $3 \leq n \leq 6$, extending the bounds from Sections~\ref{nessaryconditionKmPn}--\ref{conboundKmPn}.
	\subsection*{A.1 Self-identifying code of $K_3 \times P_n$ for $3 \leq n \leq 6$}\label{K3Pn}
	
	\textbf{Theorem A.1.}
	For $K_3 \times P_n$,
	\[
	\gamma^{\mathrm{SID}}(K_3 \times P_n) = 
	\begin{cases}
		9 & n=3, \\
		12 & n=4,5, \\
		14 & n=6.
	\end{cases}
	\]
	
	\begin{proof}
		Lower bounds follow from Lemma~\ref{c0cn-1} and Corollary~\ref{c1cn-2}. Upper bounds are achieved by explicit constructions (Figure~\ref{m=3n=3456}),
		\begin{itemize}
			\item $S = V(K_3 \times P_3)$ for $n=3$
			\item $S = V(K_3 \times P_4)$ for $n=4$
			\item $S = C_0 \cup C_1 \cup C_3 \cup C_4$ for $n=5$
			\item $S = C_0 \cup C_1 \cup C_4 \cup C_5 \cup \{(v_2,2),(v_2,3)\}$ for $n=6$
		\end{itemize}
		Verification uses Observation~\ref{factKmPn}.
	\end{proof}
	
	\begin{figure}[H]
		\centering
		\subfigure[$n$=3]{\label{m=3n=3}\includegraphics[width=1.15cm]{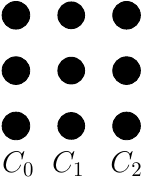}}
		\hspace{0.25in}
		\subfigure[$n=4$]{\label{m=3n=4}\includegraphics[width=1.5cm]{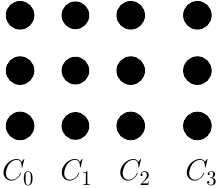}}
		\hspace{0.25in}
		\subfigure[$n=5$]{\label{m=3n=5}\includegraphics[width=1.8cm]{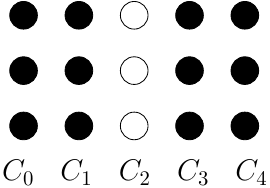}}    
		\hspace{0.25in}
		\subfigure[$n=6$]{\label{m=3n=6}\includegraphics[width=2.2cm]{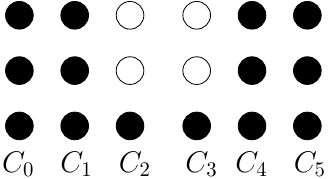}}
		\caption{A self-identifying code of $K_3\times P_n$ for $n=3,4,5,6$.}\label{m=3n=3456}
	\end{figure}
	
	\subsection*{A.2 Self-identifying code of $K_m \times P_3$ and $K_m \times P_4$ for $m \geq 4$}\label{KmP34}
	
	\textbf{Theorem A.2.}\label{thmn=34m=4...}
	For $m \geq 4$,
	\[
	\gamma^{\mathrm{SID}}(K_m \times P_n) = 
	\begin{cases}
		2m + 3 & n=3, \\
		2m + 6 & n=4.
	\end{cases}
	\]
	
	\begin{proof}
		Lower bounds follow from Lemma~\ref{c0cn-1} and Corollary~\ref{c1cn-2}. Upper bounds are achieved by explicit constructions (Figure~\ref{n=34m=4...}),
		\begin{itemize}
			\item $S = C_0 \cup C_2 \cup \{(v_0,1),(v_1,1),(v_2,1)\}$ for $n=3$,
			\item $S = C_0 \cup C_3 \cup \{(v_0,1),(v_1,1),(v_2,1),(v_0,2),(v_1,2),(v_2,2)\}$ for $n=4$.
		\end{itemize}
		Verification uses Observation~\ref{factKmPn}.
	\end{proof}
	
	\begin{figure}[H]
		\centering
		\subfigure[$n$=3]{\label{n=3m=4...}\includegraphics[width=1.3cm]{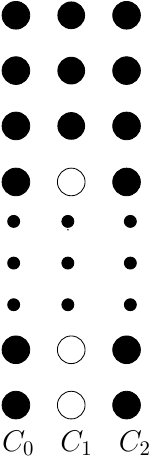}}
		\hspace{0.25in}
		\subfigure[$n=4$]{\label{n=4m=4...}\includegraphics[width=1.5cm]{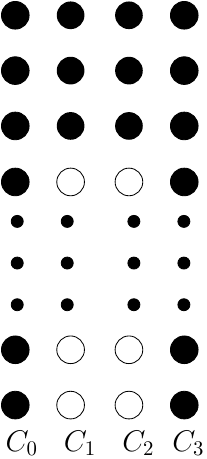}}
		\caption{A self-identifying code of $K_m\times P_n$ for $n\in \{3,4\}$ and $m\ge 4$.}\label{n=34m=4...}
	\end{figure}
	
	\subsection*{A.3 Self-identifying code of $K_m \times P_5$ and $K_m \times P_6$ for $m \geq 4$}\label{KmP56}
	
	\textbf{Theorem A.3.}\label{thmn=45m=4...}
	For $m \in \{4,5\}$,
	\[
	\gamma^{\mathrm{SID}}(K_m \times P_n) = 
	\begin{cases}
		2m + 6 & n=5, \\
		4m & n=6.
	\end{cases}
	\]
	
	\begin{proof}
		Lower bounds follow from Lemma~\ref{c0cn-1} and Corollary~\ref{c1cn-2}. Upper bounds are achieved by explicit constructions (Figure~\ref{n=56m=45}),
		\begin{itemize}
			\item $S = C_0 \cup C_4 \cup \{(v_0,1),(v_1,1),(v_2,1)\} \cup \{(v_{m-3},3),(v_{m-2},3),(v_{m-1},3)\}$ for $n=5$,
			\item $S = C_0 \cup C_5 \cup \{(v_0,1),(v_1,1),(v_2,1),(v_0,4),(v_1,4),(v_2,4)\} \cup \bigcup_{i=3}^{m-1} \{(v_i,2),(v_i,3)\}$ for $n=6$
		\end{itemize}
	\end{proof}
	\begin{figure}[H]
		\centering
		\subfigure[$n$=5, $m$=4]{\label{n=5m=4}\includegraphics[width=2.2cm]{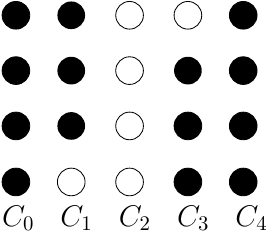}}
		\hspace{0.25in}
		\subfigure[$n$=$m$=5]{\label{n=5m=5}\includegraphics[width=2cm]{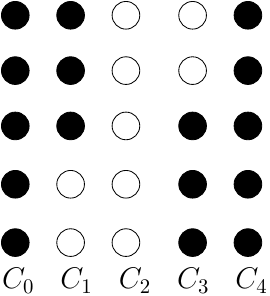}}
		\hspace{0.25in}
		\subfigure[$n$=6, $m$=4]{\label{n=6m=4}\includegraphics[width=2.3cm]{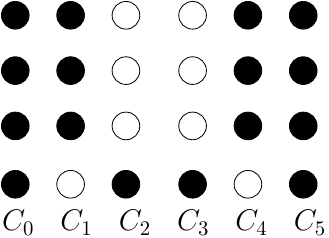}}
		\hspace{0.25in}
		\subfigure[$n$=6, $m$=5]{\label{n=6m=5}\includegraphics[width=2.4cm]{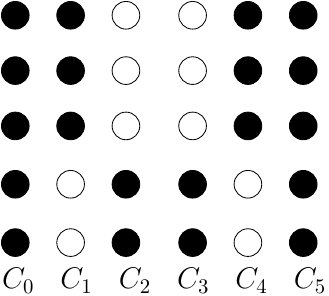}}
		\caption{A self-identifying code of $K_m\times P_n$ for $n\in\{5,6\}$ and $m\in\{4,5\}$.}\label{n=56m=45}
	\end{figure}
	
	\textbf{Theorem A.4.}\label{thmn=56m=6...}
	For $m \geq 6$,
	\[
	\gamma^{\mathrm{SID}}(K_m \times P_n) = 
	\begin{cases}
		3m & n=5, \\
		3m + 6 & n=6.
	\end{cases}
	\]
	
	\begin{proof}
		Lower bounds follow from Lemma~\ref{c0cn-1} and Corollary~\ref{c1cn-2}. Upper bounds are achieved by explicit constructions (Figure~\ref{n=56m=6...}),
		\begin{itemize}
			\item $S = C_0 \cup C_4 \cup \{(v_0,1),(v_1,1),(v_2,1)\} \cup \bigcup_{i=3}^{m-1} \{(v_i,3)\}$ for $n=5$,
			\item $S = C_0 \cup C_5 \cup C_2 \cup \{(v_0,1),(v_1,1),(v_2,1),(v_0,4),(v_1,4),(v_2,4)\}$ for $n=6$
		\end{itemize}
	\end{proof}
	
	\begin{figure}[H]
		\centering
		\subfigure[$n$=5, $m \geq $6]{\label{n=5m=6...}\includegraphics[width=2.2cm]{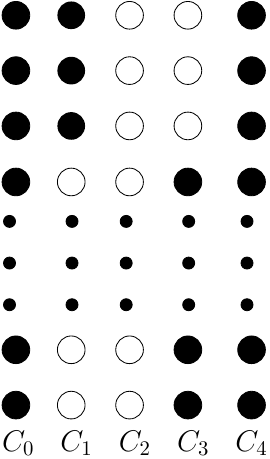}}
		\hspace{0.25in}
		\subfigure[$n$=6, $m \geq$ 6]{\label{n=6m=7...}\includegraphics[width=2.3cm]{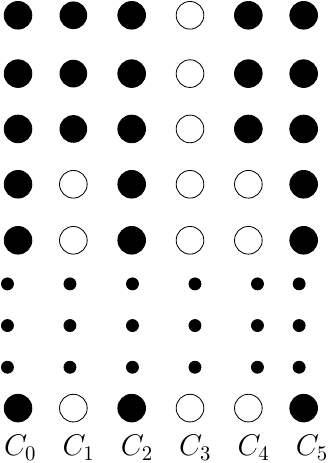}}
		\caption{A self-identifying code of $K_m\times P_n$ for $n=5,6$ and $m\ge 6$.}\label{n=56m=6...}
	\end{figure}
\end{document}